\documentclass[11pt,oneside,reqno]{amsart}
      \usepackage{amssymb,enumerate,dsfont,url}

      \theoremstyle{plain}
      \newtheorem{theorem}{Theorem}[section]
      \newtheorem{lemma}[theorem]{Lemma}
      
	  \newtheorem{proposition}[theorem]{Proposition}
             
      \theoremstyle{definition}
      \newtheorem{definition}[theorem]{Definition}      
      \theoremstyle{remark}
      \newtheorem{remark}[theorem]{Remark}

	\DeclareMathOperator{\tr}{tr}

      \makeatletter
      \def\@setcopyright{}
      
      \def\serieslogo@{}
      \makeatother

\usepackage{geometry}
\usepackage{marginnote}
\usepackage{verbatim}
\usepackage{mathtools}

\begin{document}
   \author{Kristina Schubert}
   \address{Inst. for Math. Stat., Univ. M\"unster, Orl\'{e}ans-Ring 10, 48149 M\"unster, Germany}
   \email{kristina.schubert@uni-muenster.de}

  \title[Random Matrices with Independent Skew-Diagonals]{Spectral Density for Random Matrices with Independent Skew-Diagonals}

 \begin{abstract}
We consider the empirical eigenvalue distribution of random real symmetric matrices with stochastically independent skew-diagonals and study its limit if the matrix size tends to infinity. We allow   correlations between entries on the same  skew-diagonal and we distinguish between two types of such correlations, a rather weak and a rather strong one. For weak correlations the limiting distribution is Wigner's semi-circle distribution; for strong correlations it is the free convolution of the semi-circle distribution and the limiting distribution for random Hankel matrices.
\end{abstract}





   \maketitle

\setcounter{equation}{0}

\section{Introduction}
Wigner's semi-circle law is possibly the most famous principle in random matrix theory. It states that, for various random matrix ensembles, the empirical eigenvalue distribution tends to a universal limit if the matrix size tends to infinity. 
The most prominent representatives of ensembles following this law are  the three classical Gaussian Ensembles (GOE, GUE and GSE) and the more general Wigner ensembles. The latter  maintain    
the independence of the matrix entries of the classical Gaussian ensembles, but allow other distributions than the normal distribution. 

The proof of Wigner's semi-circle law started with the  pioneering works of  Wigner himself \cite{Wig1955,Wig1958} and experienced many generalizations, e.g.~by  Arnold \cite{Arnold}. This proof was recently accomplished  for Wigner ensembles under rather mild regularity assumptions in a series of papers by Tao and Vu \cite{TaoVu,TaoVu2} and Erd{\H{o}}s et al.~(see  e.g.~the survey \cite{Erdos2}). 
The  law   states that  the empirical eigenvalue distribution   in Wigner ensembles converges weakly, in probability,   to a non-random distribution,  if the matrix size   tends to infinity. The limiting  distribution is then given by the  semi-circle distribution. 
 
The appearance of this distribution in the limit of large matrix sizes, regardless of the distribution of the matrix entries, hints at a phenomenon  often encountered in the study of random matrices, called universality. In general, the universality principle describes the fact that several limiting distributions of eigenvalue statistics do not depend on the details of the underlying random matrix ensemble.

However, one may ask to which extent the independence of the matrix entries is necessary for the limiting spectral density to be the semi-circle and how this limit changes if not all matrix entries are independent. 
Hence, matrices with a dependence structure of some kind have attracted attention over the last years  and were e.g.~studied in \cite{MKV,GNT,GT,Kirsch,KP,SSB}.  
 %
 
One possible approach to matrices with correlated real-valued entries is to allow that entries on the same (skew-)diagonal are correlated, while entries on different \mbox{(skew-)}diagonals are independent. The most distinct type of such correlations is met by random Toeplitz matrices $T_n$ and random Hankel matrices $H_n$, given by 
\begin{equation*}
T_n \coloneqq [X_{|i-j|}]_{1\leq i,j \leq n},
\quad H_n \coloneqq  [X_{i+j-1}]_{1\leq i,j \leq n}
\end{equation*}
for real-valued random variables $X_0,\ldots, X_{2n-1}$.
These matrices appear e.g.~as 
autocovariance matrices in time series analysis and as information
matrices in a polynomial regression model \cite{Bai}. The study of random Toeplitz and Hankel matrices was proposed by Bai  \cite{Bai} and later  addressed by Bryc, Dembo and Jiang   \cite{BDJ}. It was shown that if  $X_0,\ldots,X_{2n-1}$ have variance one, the  empirical eigenvalue distributions  of $T_n$ and $H_n$  converge weakly with probability one to  non-random limits. In particular, both limits differ from the semi-circle distribution and hence oppose the universality results for classical ensembles with independent entries.
Starting from the results for Toeplitz matrices, matrices with independent diagonals have been studied by Friesen and L\"owe \cite{FL2, FL1}. 

In this paper, we consider the empirical eigenvalue distribution of real symmetric random matrices with independent skew-diagonals instead of independent diagonals. We assume that  each matrix entry is centered with the same variance and that the $k$-th moments  are uniformly bounded for all matrix sizes. 
Our main results state that the empirical eigenvalue distribution converges weakly, with probability one, to a non-random distribution. 
Here, we distinguish two ensembles, one allowing for weak correlations, the other one allowing for strong  correlations.

By weak correlations, we mean that the covariance of two entries on the same skew-diagonal  depends on their distance only and   decays sufficiently fast. In this case, we show that the limiting spectral density is given by the semi-circle. 
When we consider  a type of rather strong correlations between entries on the same skew-diagonal, we  assume that  these correlations depend on the matrix size only and converge as the matrix size tends to infinity. 
Here, the limiting spectral distribution is   given by a combination  of the semi-circle distribution and the limiting distribution for Hankel matrices. Both distributions are obtained in special cases, where the correlation of two entries on the same skew-diagonal tends to zero or to one respectively. 
These results are analogue to the ones obtained for ensembles of  matrices with independent diagonals   in \cite{FL2,FL1}. 

In our proof,  we can follow the basic ideas of \cite{BDJ, FL2, FL1},  while several  technical difficulties arise for the current ensembles, preventing  the results of \cite{FL2,FL1} from being directly applicable. 
To gain insight into these  difficulties, we note that
the basic tool in the study of the expected empirical distribution, both in  the setting of \cite{FL2,FL1} and in the current setting, is the method of moments. Hence, for 
 an $n \times n$ random matrix $X_n=\frac{1}{\sqrt{n}}(a_n(i,j))_{1 \leq i,j \leq n}$ we need to consider the $k$-th  expected moment of the empirical eigenvalue distribution, which is given by
 \begin{eqnarray}
 \label{k_moment}
\frac{1}{n} \mathbb E \left[ \tr (X_n^k)\right]  
=\frac{1}{n^{\frac{k}{2} + 1}} \sum_{p_1,\ldots, p_k =1}^n \mathbb E\left[ a_n(p_1,p_2) a_n(p_2,p_3) \ldots a_n(p_k,p_1)\right].
\end{eqnarray}
To calculate (\ref{k_moment}), we have to consider the (in)dependence structure of the matrix entries: For the ensembles in \cite{FL2,FL1}  two entries $a_n(i,j)$ and $a_n(i',j')$ are stochastically independent if neither they nor $a_n(j,i)$ and $a_n(i',j')$ are on the same  diagonal, i.e. if 
\begin{equation}
\label{ind_FL}
|i-j|\neq|i'-j'|.
\end{equation} 
For the current ensembles with independent skew-diagonals, entries $a_n(i,j)$ and $a_n(i',j')$  are  stochastically independent if  
\begin{equation}
\label{ind_current}
i+j\neq i'+j'.
\end{equation} 
Although the defining relations (\ref{ind_FL}) and  (\ref{ind_current}) appear quite similar, the implications are more involved leading to two major difficulties. 

Firstly, a key observation for the calculation of (\ref{k_moment})  in  \cite{FL2,FL1} are the general equations $ \sum_{i=1}^k (p_i-p_{i+1}) =0$ (as $k+1$ is identified with 1) and $\sum_{i=l}^m (p_i-p_{i+1}) =p_l - p_{m+1}$.
 However, the appearance of the differences $p_i-p_{i+1}$  
makes these equations well applicable to  ensembles with independence structure (\ref{ind_FL}) in contrast to ensembles with independence structure (\ref{ind_current}).  This leads to the  most prominent difference, Lemma \ref{lemma_B_n}, which states that a certain quantity vanishes as the matrix size $n$ tends to infinity  rather than being zero for all $n$ as  in the analogue result in \cite{FL1}.

Secondly, the usual symmetry condition $a_n(i,j)=a_n(j,i)$ affects the matrix ensembles with independent diagonals and independent skew-diagonals in different ways. 
An $n \times n$ matrix is built from 
$n$ independent families of random variables in the first case (one for each diagonal in the upper triangular matrix) and from $2n-1$ independent families in the second case (one for the upper half of each skew-diagonal) respectively.
This apparent `higher degree of independence' does however not favor the analysis. Instead, when in the course of the calculations we consider pairs of matrix entries on the same skew-diagonal, say $a_n(i,j)$, $a_n(i+c, j-c)$,  we have to explicitly treat such pairs, where the symmetry of the matrix implies $a_n(i,j)=a_n(i+c, j-c)$, i.e. $c=j-i$ (see Lemma \ref{crossing_prop} onward).

For the convenience of the reader, we  follow the line of arguments presented in  \cite{FL2,FL1}. We adapt the proofs to the current ensembles and insert  new ideas when necessary. 
More detailed comments on the differences between the methods in  \cite{FL2,FL1} and the methods in this paper are given in Remark \ref{remark_diff_toeplitzcase}.

Matrices from the ensembles considered in this paper and in \cite{FL2,FL1} can actually be generated in several ways.  
Ensembles with weak  correlations along the (skew-)diagonals can e.g.~be built from independent families of stationary Gaussian Markov processes with mean zero and variance one. 
One can also fill the independent (skew-)diagonals with random variables from the Curie-Weiss model with inverse temperature $\beta>0$. These exhibit the required strong correlations and for $\beta>1$ the limiting law is the described combination of the semi-circle distribution and the limiting law for Toeplitz matrices or for Hankel matrices respectively. The  phase transition at $\beta=1$ in the Curie-Weiss model corresponds to the fact that for $\beta \leq 1$ the limiting law is the semi-circle distribution. Details on these examples can be found in  \cite{FL2,FL1}.

This paper is organized as follows: In Section \ref{sec_mainresult} we introduce our model of matrices with independent skew-diagonals and two different conditions for the correlations on the same skew-diagonal. Moreover, we state our main results (Theorem~\ref{MainTheorem_2} and  \ref{MainTheorem_1}) about the convergence of the empirical eigenvalue distribution. In Section \ref{sec_combinatoric}, following the ideas of  \cite{BDJ}, we introduce the notion of partitions to model the dependence structure of the matrix entries and derive an intermediate result for the expected $k$-th moment of the empirical eigenvalue distribution; this calculation is completed in Section \ref{sec_kthmoment}. In Section \ref{sec_completing_proofs} we extend the convergence of the expected $k$-th moment of the empirical eigenvalue distribution to the required weak convergence with probability one. In the case of strong correlations we further show some results for the limiting distribution. In particular, we show that the limiting distribution is the free convolution of the semi-circle distribution and the limiting distribution for Hankel matrices. 
%
\section{Main Results}\label{sec_mainresult}
We introduce our ensembles of random matrices and then state our main theorems. For $n \in \mathbb N$ we let $a_n(p,q)$ with $1 \leq p \leq q\leq n$ denote real random variables. 
We will consider the eigenvalues of the symmetric random $n \times n$ matrix $X_n$ obtained from $a_n(p,q)_{1 \leq p \leq q\leq n}$ by rescaling
			\begin{align*}
			X_n(p,q)&= \frac{1}{ \sqrt{n}} a_n(p,q), \quad  1\leq p\leq q\leq n,  
		\\ X_n(p,q)&=X_n(q,p), \quad  \quad \, \, 1\leq q< p	\leq n. 
			\end{align*}
We suppose that the random variables $a_n(p,q)$ are centered with variance one  and the $k$-th moments are uniformly bounded.
Moreover, we assume that the skew-diagonals are independent. 
Technically, these assumptions read:
      \begin{enumerate}
      \item[(A1)]  $$\displaystyle \mathbb E (a_n(p,q))=0, \quad \mathbb E ((a_n(p,q)^2)=1, \quad  1 \leq p\leq q\leq n $$
      \item[(A2)] $$\displaystyle m_k \coloneqq \sup_{n \in \mathbb N} \max_{1 \leq p\leq q \leq n} \mathbb E\left(|a_n(p,q)|^k \right
			)<\infty$$
      \item[(A3)] The families $\{ a_n(p,q): p+q=r\}$ are independent for $r=2,3,\ldots,2n.$ \\
      \end{enumerate}
We will consider two types of matrices with different conditions on the covariances of entries from the same skew-diagonal.  On the one hand, we assume that these covariances depend on the distance of the entries and decay sufficiently fast. On the other hand, we assume that  these covariances  are constant (only depending on $n$) and that they converge as $n\to \infty$. These conditions are 
       \begin{itemize}
       \item[(C1)]There exists a function  $c_n:\mathbb N \to\mathbb R$ such that \\
                     \begin{enumerate}
       \item[(i)] for $p\leq q, r\leq s$ with $p+q=r+s$ we have
       $$\quad  |\text{Cov}(a_n(p,q),a_n(r,s))| = c_n (|p-r|) =c_n(|q-s|).$$
       \item[(ii)] For $n \in \mathbb N$ we have  
			$$\displaystyle \sum_{\tau=0}^{n-1} c_n(\tau) =o(n).$$
       \end{enumerate}
       \item[(C2)] There exists $(c_n)_{n\in \mathbb N}$ such that  for all  $p,p',q,q' \in \{1,\ldots, n\}$ with $p+q=p'+q'$ and $(p,q)\notin \{ (p',q'), (q',p')\}$ we have:
       $$\text{Cov} (a_n(p,q), a_n(p',q'))=c_n.$$
			Moreover, the limit
       $ c\coloneqq \lim_{n \to \infty} c_n<\infty$
       exists.
       \end{itemize}
			\begin{remark}
			If condition (C2) is satisfied, we have 
			$0\leq c \leq 1$ (see Remark 2.1 in \cite{FL2}).
			Indeed, $0\leq c$  is a consequence of 
			$$ 0 \leq \mathbb V\left( \sum_{p=1}^n  a_n(p,p) \right)=n + n(n-1) c_n$$
			and $c_n\leq 1$ is a consequence of H\"older's inequality. 
			\end{remark}
For  the ordered eigenvalues of a matrix $X_n$, denoted by $\lambda_1^{(n)} \leq \ldots \leq \lambda_n^{(n)}$, we introduce the  empirical eigenvalue distribution
	\begin{equation*}
	\mu_n(
	X_n)\coloneqq\frac{1}{n} \sum_{k=1}^n \delta_{\lambda_k^{(n)}}.
	\end{equation*}
Our main theorems then state the weak convergence of $\mu_n$ under condition (C1) resp.~under (C2). If (C1) is satisfied, the limiting distribution is Wigner's semi-circle distribution.  
\begin{theorem}\label{MainTheorem_2}
Suppose that (A1), (A2), (A3) and (C1) are satisfied. Then, with probability one, $\mu_n$ converges weakly to the standard semi-circle distribution $\mu$ with density
\begin{equation}\label{semi_circ}
\frac{d\mu}{dx}\frac{1}{2\pi} \sqrt{4-x^2}\, \chi_{[-2,2]}(x).
\end{equation}
\end{theorem}
\begin{theorem}\label{MainTheorem_1}
Suppose that (A1), (A2), (A3) and (C2) are satisfied. Then, with probability one, $\mu_n$ converges weakly to a non-random probability measure $\nu_c$. The limiting measure $\nu_c$ does not depend on the distribution of the random variables $a_n(p,q)$. 
\end{theorem}
If condition (C2) is satisfied,  we can show further results for  $v_c$.
\begin{theorem}
\label{theo_limiting_distr}
In the situation of Theorem \ref{MainTheorem_1}, the limiting measure $v_c,$ $0 \leq c \leq 1,$ 
is the free convolution of the measures $v_{0,1-c}$ and $v_{1,c}$, we write $v_c=v_{0,1-c} \boxplus v_{1,c}$. Here, $v_{0,1-c} $ denotes the rescaled semi-circle   with variance $1-c$ and  $v_{1,c}$ the rescaled measure for Hankel matrices $\gamma_H$ with variance c as derived in \cite{BDJ}.
Moreover, $v_c$ is a symmetric measure. If $c > 0$, $v_c$ 
has an unbounded support, and if $0 \leq  c < 1$, its density is smooth.
\end{theorem}
\begin{remark}
Here,   neither the notion of free probability nor  of the  free convolution is  introduced and the reader is referred to \cite{NicaSpeicher} for details on this topic. 
We observe that the result of Theorem \ref{theo_limiting_distr}   is in good accordance with the semi-circle law and the results of \cite{BDJ}. Indeed, 
on the one hand, matrices from the GOE satisfy (C2) with $c=0$, indicating no dependence at all, and hence $v_0$ has to be the semi-circle distribution. On the other hand, random Hankel matrices satisfy (C2) with $c=1$ and thus we conclude that $v_1$ equals $\gamma_H$ in \cite{BDJ}. This also becomes apparent in our calculations of the moments of $v_c$, where it turns out that the moments of $v_0$ are exactly given by the moments of the semi-circle distribution and the moments of $v_1$ are given by the moments of $\gamma_H$ (see Remark \ref{Remark_lim_distr}).
\end{remark}
%
%
   \section{Preliminaries, Notation and Combinatorics} \label{sec_combinatoric}
   Following the suggestions of \cite{BDJ,FL2,FL1}, we will prove Theorems \ref{MainTheorem_2} and  \ref{MainTheorem_1} by the method of moments.  
   In order to calculate  the $k$-th moment of the expected empirical distribution $\mu_n$, we introduce some notation and some combinatorial arguments that will repeatedly be used in the proofs.
We can  expand 
\begin{eqnarray*}
\mathbb E\left[ \int x^k d\mu_n(X_n) \right] &=& \frac{1}{n} \mathbb E \left[ \tr (X_n^k)\right]  \\
&=& \frac{1}{n^{\frac{k}{2} + 1}} \sum_{p_1,\ldots, p_k =1}^n \mathbb E\left[ a_n(p_1,p_2) a_n(p_2,p_3) \ldots a_n(p_k,p_1)\right].
\end{eqnarray*}
To simplify the notation we set 
\begin{equation*}
\tau_n(k)\coloneqq \{(P_1,\ldots, P_k): P_j=(p_j,q_j) \in \{1,\ldots,n\}^2, q_j=p_{j+1} \}.
\end{equation*}
Hence, in the expansion of the $k$-th moment we  write 
$ a_n(P_i) \text{ instead of } a_n(p_i,p_{i+1})$ for $P_i=(p_i,p_{i+1})$ and we sum over all $(P_1,\ldots, P_k) \in\tau_n(k)$.
Throughout this paper, we identify $k+1$  with $1$.

In order to display the dependence structure of the matrix entries, we use the notion of partitions as suggested by   \cite{BDJ}. We want to express that $a_n(P_i)$ and $a_n(P_j)$ are not stochastically independent, i.e.~they denote entries on the same skew-diagonal, by  $i$ and $j$ being  in the same block of the corresponding partition.
More precisely, for a partition $\pi$ of $\{1,\ldots,k\}$ we call $(P_1,\ldots, P_k) \in\tau_n(k)$ a $\pi$\textit{-consistent sequence} if 
\begin{equation} \label{equivalence_relation}
p_i+q_i = p_j+q_j\quad \Leftrightarrow \quad  i\sim j.
\end{equation}
We write $i\sim j$ instead of $i\sim_{\pi} j$ if the partition $\pi$
can be recovered from the context.
With \begin{equation*}
\mathcal{P}(k) \coloneqq \{\pi : \pi \text{ is a partition of } \{1,\ldots, k\}\},
\end{equation*}
\begin{equation*}
S_n(\pi) \coloneqq \{(P_1,\ldots,P_k) \in \tau_n(k) : (P_1,\ldots,P_k) \text{ is } \pi \text{-consistent} \}, \quad \pi \in \mathcal{P}(k),\end{equation*}
we have 
\begin{eqnarray}
&&\frac{1}{n^{\frac{k}{2}  + 1}} \sum_{p_1,\ldots, p_k =1}^n \mathbb E\left[ a_n(p_1,p_2) a_n(p_2,p_3) \ldots a_n(p_k,p_1)\right]\nonumber \\
&=& \frac{1}{n^{\frac{k}{2}  + 1}} \sum_{\pi \in \mathcal{P}(k)} \sum_{(P_1,\ldots,P_k) \in S_n(\pi)} \mathbb E\left[ a_n(P_1) a_n(P_2) \ldots a_n(P_k)\right]. \label{partition:1}
\end{eqnarray}
Here, we used that the sets $S_n(\pi), \pi \in \mathcal{P}(k)$, are a partition of $\tau_n(k)$.
In the next subsection we argue that it is only the pair partitions that give a non-vanishing contribution in (\ref{partition:1}).

\subsection{Reduction to pair partitions} We observe that in  (\ref{partition:1}) terms corresponding to partitions $\pi$ with more than $\frac{k}{2}$ blocks, i.e.~$\# \pi > \frac{k}{2}$,  vanish, since in this case there is a partition block with a single element $i$. Hence, $a_n(P_i)$ is independent of all the other $a_n(P_j), j\neq i$ if $(P_1,\ldots,P_k) \in S_n(\pi)$ and as $a_n(P_i)$ is centered (see (A1)) the respective term equals zero. 
We claim that for partitions with less than $k/2$ blocks we have
\begin{equation}
\frac{\# S_n(\pi)}{n^{\frac{k}{2}  + 1}}=o(1), \quad    \# \pi < \frac{k}{2}.
\label{partition:2}
\end{equation}
This can be seen from the following combinatorial arguments used to determine the number of possibilities to construct an element $(P_1,\ldots, P_k) \in S_n(\pi)$:
\begin{itemize}
\item Once $P_i=(p_i,p_{i+1})$ is fixed, the pair $P_{i+1}=(p_{i+1},p_{i+2})$ is  determined by the choice of $p_{i+2}$.
\item We start with the choice of $P_1=(p_1,p_2)$, for which there are at most $n^2$ possibilities.
\item We proceed sequentially to determine $P_2, P_3, \ldots  $ as follows: To determine $P_i$, if $i$ is in the same block of $\pi$ as some preceding index $j \in\{1,\ldots, i-1\}$, there is no choice left, as the indices need to satisfy $p_j+p_{j+1}=p_i+p_{i+1}$, where $p_j,p_{j+1},p_i$ are already known. Otherwise, there are at most $n$ possible choices. Once $P_1$ is fixed, there are  $n$ possibilities for each `new' partition block, i.e.~$\# \pi -1$ times. 
\end{itemize}
Hence, we obtain $\# S_n(\pi)\leq  n^2 \cdot n^{\#\pi -1}= n^{\#\pi +1}$, proving  the claim in (\ref{partition:2}). 
Together with 
\begin{equation*}
|\mathbb E [a_n(P_1) \ldots a_n(P_k)]| \leq  \prod_{i=1}^k \left[  \mathbb E|a_n(P_i)|^k\right]^{\frac{1}{k}} \leq m_k,
\end{equation*}
 which is a consequence of the uniform boundedness of the moments in (A2) and    H\"older's inequality, we obtain  
 \begin{equation}\label{partition:3}
\frac{1}{n} \mathbb E \left[ \tr X_n^k \right]= 
\frac{1}{n^{\frac{k}{2}  + 1}} \sum_{\pi \in \mathcal{P}(k) \atop \# \pi=\frac{k}{2}} \sum_{(P_1,\ldots,P_k) \in S_n(\pi)} \mathbb E\left[ a_n(P_1) a_n(P_2) \ldots a_n(P_k)\right] +o(1).
 \end{equation}
 In particular, the odd moments vanish. 
Moreover, it suffices to consider pair partitions $\pi$ in (\ref{partition:3}), i.e.~partitions where each block has exactly two elements. Indeed,  partitions with $\# \pi=\frac{k}{2}$ that are not pair partitions contain a block with a single element and do hence not contribute to (\ref{partition:3}) by the same reasoning that allowed us to exclude partitions with $\# \pi > \frac{k}{2}$.
 %
 %
 %
 \subsection{Partitions and combinatorics}  
We sum up some combinatorial arguments  about partitions and introduce some notation. 
A recurring combinatorial consideration is the following:
Suppose we want to determine the number of possible vectors $(P_1,\ldots, P_k) \in S_n(\pi)$ or parts of this vector for a given pair partition $\pi$. We write $P_i=(p_i,p_{i+1}), i=1,\ldots,k$ and state two counting principles:
\begin{enumerate}
\item[(CP1)]
Assume that only $p_i$ is already fixed for some $i \in \{1,\ldots,k\}$ and we want to choose values for $p_{i+1},\ldots, p_{j+1}$ for some $j>i$ (i.e.~we want to fix $P_i,\ldots,P_j$). We start with the choice of $p_{i+1}$, for which there are $n$
possibilities. Then, we proceed sequentially to fix $p_{i+2},p_{i+3},\ldots$ and for each $p_l$ we have $n$ choices if $l$ is not in the same block as any of the $i,i+1,\ldots,l-1$. Otherwise there is no choice and $p_l$ is already determined by the requirement of $\pi$-consistency. Hence, if $r$ denotes the number of blocks that are occupied by $\{i,\ldots,j\}$,
we have $n^r$ possibilities to fix $P_i,\ldots,P_j$ (for given  $p_i$).
\item[(CP2)]
As in (CP1) we want to choose values for $p_{i+1},\ldots, p_{j+1}$ for some $j>i$ (i.e.~we want to fix $P_i,\ldots,P_j$) and we assume that in addition to $p_i$  some values $P_{i_1},\ldots, P_{i_l}$ with $\{i_1,\ldots,i_l \}\cap \{i,\ldots, j\}=\emptyset$ have already been fixed. Again, we  start with the choice of $p_{i+1}$. If $i+1$ is not equivalent to any of the $i_1,\ldots,i_l$, there are $n$ possibilities, otherwise $p_{i+1}$ is already fixed by the  $\pi$-consistency. For $p_{i+2}$ there are $n$ possibilities if $i+2$ is not equivalent to any of the indices $i+1,i_1,\ldots,i_{l}$, otherwise there is no choice. Proceeding sequentially, we have 
 $n^{r-s}$ possibilities to fix $P_i,\ldots,P_j$  if $r$ denotes the number of partition blocks that are occupied by $\{i,\ldots,j\}$ and $s$ denotes the number of indices in $\{i,\ldots,j\}$ that are equivalent to any of the $i_1,\ldots,i_l$. In other words, $r-s$ is the number of partition blocks occupied by  $\{i,\ldots,j\}$, which have an empty intersection with $\{i_1,\ldots,i_l \}$.
\end{enumerate}
After stating these counting principles, which we will use in the later proofs, we sum up some facts about pair partitions.
We distinguish between \textit{crossing} pair partitions   and \textit{non-crossing}  pair partitions. A pair partition is said to be crossing if there exist indices $i<j<l<m$ with $i\sim l$ and $j \sim m$. 
We set 
\begin{align*}
\mathcal{PP}(k)&\coloneqq \{ \pi \in \mathcal{P}(k): \pi \text{ is a pair partition}\},\\
\mathcal{CPP}(k)& \coloneqq \{ \pi \in \mathcal{P}(k): \pi \text{ is a crossing pair partition}\},\\
\mathcal{NCPP}(k)& \coloneqq \{ \pi \in \mathcal{P}(k): \pi \text{ is a non-crossing pair partition}\}.
\end{align*}
For a non-crossing pair partition $\pi \in \mathcal{NCPP}(k)$ and $(P_1,\ldots,P_k) \in S_n(\pi)$ we have
\begin{itemize}
\item[(NC1)]
There exist indices $i,j \in \{1,\ldots, k\}$ with $i\sim j $ and $j=i+1$.
\item[(NC2)]
If $i\sim j$ and $j=i+1$, we have $a_n(P_i)=a_n(P_j)$ and hence we have  $\mathbb E [a_n(P_i)a_n(P_j)]=1$. This is due to the fact that $i \sim i+1$ implies that for $P_i=(p_i,p_{i+1})$ and $P_{i+1}=(p_{i+1}, p_{i+2})$  we have $p_i=p_{i+2}$. Hence, by the symmetry of the considered matrix we have $a_n(P_i)=a_n(p_{i+2},p_{i+1})=a_n(P_{i+1})$.
\item[(NC3)] If  $i\sim j$ and $j=i+1$, the sequence $P'\coloneqq(P_1, \ldots, P_{i-1}, P_{i+2}, \ldots, P_k)$ is  in $\tau_n(k-2)$. This is a consequence of (NC2).
\item[(NC4)] The number of non-crossing pair partitions on $k$ elements is given by  (see e.g.~Lemma 8.9 in \cite{Speicher})
\begin{equation*}
\# \mathcal{NCPP}(k)=C_{\frac{k}{2} },
\end{equation*}
where $C_k\coloneqq\frac{1}{k+1} \binom{2k}{k}$ denotes the $k$-th Catalan number.
\end{itemize}
With these notations and results about pair partitions we can continue with the calculation of (\ref{partition:3}).
\section{Calculating the Expected $k$-th Moment of $\mu_n$}\label{sec_kthmoment}
We return to the expected $k$-th moment of the spectral distribution given in (\ref{partition:3}).  In the following lemma, we show that summing over all non-crossing pair partitions in (\ref{partition:3}) equals $C_{\frac{k}{2} }$ under both (C1) and (C2). The contribution of the crossing partitions  is studied in subsection \ref{subsec_C1} for condition (C1)  and  in subsection \ref{subsec:C1} for condition (C2).
\begin{lemma}[cf.~Lemma 5.2 and 5.3 in \cite{FL2}]\label{lemma_NCPP}
Under condition (C1) and (C2) we have for $k\in \mathbb N$ even
\begin{align*}
\frac{1}{n} \mathbb E \left[\tr X_n^k \right]= C_{\frac{k}{2} } +
\frac{1}{n^{\frac{k}{2}  + 1}} \sum_{\pi \in \mathcal{CPP}(k) } \sum_{(P_1,\ldots,P_k) \in S_n(\pi)} \mathbb E\left[ a_n(P_1) a_n(P_2) \ldots a_n(P_k)\right] +o(1).
\end{align*} 
\end{lemma}
\begin{proof}
By successively applying (NC1)-(NC3) we have for any $\pi \in \mathcal{NCPP}(k)$ and $(P_1,\ldots,P_k) \in S_n(\pi)$:
\begin{equation*}
\mathbb E [a_n(P_1) \ldots a_n(P_k)]=1.
\end{equation*}
We claim that 
\begin{equation} \label{Anz_S_n}
\lim_{n \to \infty} \frac{\# S_n(\pi)}{n^{\frac{k}{2}+1}}=1, \quad \pi \in \mathcal{NCPP}(k).
\end{equation}
Let $(P_1,\ldots,P_k)$ be in $S_n(\pi)$.
According to (NC1)-(NC3) we have $ i \sim i+1$ for some $i\in \{1,\ldots,k\}$ and $P' \coloneqq(P_1, \ldots, P_{i-1}, P_{i+2}, \ldots, P_k) \in \tau_n(k-2)$. Moreover, we have 
\begin{equation*}
P' \in  S_n(\pi'),
\end{equation*}
where 
$
\pi' \coloneqq\pi \setminus \{\{i,i+1\}\} \in \mathcal{NCPP}(k-2) $ and all $l\geq i+2 $ are relabeled  to $l-2$.

Thus, all possible $(P_1,\ldots, P_k) \in S_n(\pi)$ can be constructed from a choice of $P'$ and a choice of $p_{i+1}$. For $p_{i+1}$ there are $n-\frac{k-2}{2}$ possibilities, as we have to ensure that $p_i+p_{i+1}$ does not equal any of the $(k-2)/2$ values  $p_j+ p_{j+1}$ for $j\neq i$, $j\neq i+1$.
This implies
\begin{equation*}
\frac{\# S_n(\pi)}{n^{\frac{k}{2}+1}} = \frac{\# S_n(\pi')}{n^{\frac{k}{2}}} + o(1).
\end{equation*}
The claim in (\ref{Anz_S_n}) then follows by induction and the fact that for $k=2$ we have
$\# S_n(\pi)= \{((p,q), (q,p)): p,q \in \{1,\ldots,n\}\}=n^2.$ Statement (NC4) completes the proof.
 \end{proof}
\subsection{The expected $k$-th moment of $\mu_n$ under (C1)}\label{subsec_C1}
In this subsection we assume that condition (C1) is satisfied and we show that for $k$ even  the expected $k$-th moment of $\mu_n$ is asymptotically given by $C_{\frac{k}{2}}$.
\begin{lemma}[cf.~Lemma 3.3 and Lemma 3.4 in \cite{FL1}]\label{Lemma_4}
If  (C1) is satisfied, we have for $k \in \mathbb N$ even
\begin{equation*}
\lim_{n \to \infty }\frac{1}{n} \mathbb E \left[ \tr X_n^k \right]= C_{\frac{k}{2} }. 
\end{equation*}
\end{lemma}
\begin{proof}
By Lemma \ref{lemma_NCPP} it suffices to show that for each $\pi \in \mathcal{CPP}(k)$ (recall that the number of these partitions depends on $k$ only) we have  
\begin{equation*}
\lim_{n \to \infty}\frac{1}{n^{\frac{k}{2}  + 1}}  \sum_{(P_1,\ldots,P_k) \in S_n(\pi)} \mathbb E\left[ a_n(P_1) a_n(P_2) \ldots a_n(P_k)\right] =0.
\end{equation*}
Let $\pi\in \mathcal{CPP}(k)$. We will define related partitions $\pi^{(1)},\ldots, \pi^{(r)}$ by successively deleting blocks of $\pi$ such that we arrive at some partition $\pi^{(r)} \in \mathcal{CPP}(k-2r)$, for which adjacent elements $j,j+1$ are in different blocks.  Suppose that $l \sim_{\pi} l+1$ for some $l \in \{1,\ldots,k\}$, otherwise we set $r=0, \pi^{(r)}=\pi$. Then we obtain $\pi^{(1)}$ from $\pi$ by deleting the block $\{l,l+1\}$
\begin{equation*}
\pi^{(1)}\coloneqq \pi \setminus \{\{l,l+1\}\}
\end{equation*}
and relabeling all $j\geq l+2$ to $j-2$. Hence $\pi^{(1)} \in \mathcal{CPP}(k-2)$.
Correspondingly, we delete  $P_l,P_{l+1}$ from $(P_1,\ldots,P_k)$ to obtain (see (NC3))
\begin{equation*}
(P_1, \ldots, P_k)^{(1)}\coloneqq(P_1, \ldots, P_{l-1}, P_{l+2}, \ldots, P_k)\in S_n(\pi^{(1)}).
\end{equation*}
We repeat this procedure to obtain $\pi^{(2)}$ and $(P_1, \ldots, P_k)^{(2)},\pi^{(3)}$ and  $(P_1, \ldots, P_k)^{(3)}$, $\ldots$ until we arrive at a partition $\pi^{(r)} \in \mathcal{CCP}(k-2r)$ where none of the blocks contains adjacent elements.   
Then $(P_1, \ldots, P_k)^{(r)} \in S_n(\pi^{(r)})$.
Since $\pi$ is a crossing partition,  at least two blocks of $\pi$ remain after the elimination process and  we have 
\begin{equation*}
r \leq \frac{k}{2}-2.
 \end{equation*}
Using the same arguments as in the proof of Lemma \ref{lemma_NCPP} leads to the following estimate for given $(Q_1,\ldots,Q_{k-2r}) \in \tau_n(k-2r)$: 
\begin{equation*}
\# \{(P_1,\ldots,P_k) \in S_n(\pi): (P_1,\ldots,P_k)^{(r)}=(Q_1,\ldots,Q_{k-2r}) \} \leq n^r.
\end{equation*}
By (NC2) we have for $(P_1,\ldots,P_k)^{(r)}=(Q_1,\ldots,Q_{k-2r})$
\begin{equation*}
\mathbb E[a_n(P_1) \ldots a_n(P_k)]=\mathbb E[a_n(Q_1) \ldots a_n(Q_{k-2r})]
\end{equation*}
and hence
\begin{align}
& \sum_{(P_1,\ldots,P_k) \in S_n(\pi)}| \mathbb E \left[ a_n(P_1) a_n(P_2) \ldots a_n(P_k)\right] | \nonumber  \\
 &\leq\quad  n^r  \sum_{(Q_1,\ldots,Q_{k-2r}) \in S_n(\pi^{(r)})} |\mathbb E\left[ a_n(Q_1) a_n(Q_2) \ldots a_n(Q_{k-2r})\right]|.
\label{NC_1.1}
\end{align}
We choose $i \sim_{\pi^{(r)}} i+j$ such that $j\geq 2$ is minimal. 
By H\"older's inequality we have for any $s,t$
\begin{equation*}
|\mathbb E [a_n(Q_s)a_n(Q_t)]| \leq (\mathbb E [a_n(Q_s)^2])^{1/2} (\mathbb E [a_n(Q_t)^2])^{1/2} =1
\end{equation*}
and hence, as $k$ is even, 
\begin{align*}
|\mathbb E\left[ a_n(Q_1) a_n(Q_2) \ldots a_n(Q_{k-2r})\right]|& \leq |\mathbb E\left[ a_n(Q_i) a_n(Q_{i+j})\right]|= \left|\text{Cov}(a_n(Q_i), a_n(Q_{i+j})) \right|.
\end{align*}
Inserting this estimate into (\ref{NC_1.1}) we obtain 
\begin{align*}
 &n^r  \sum_{(Q_1,\ldots,Q_{k-2r}) \in S_n(\pi^{(r)})} \left|\mathbb E\left[ a_n(Q_1) a_n(Q_2) \ldots a_n(Q_{k-2r})\right]\right|
\\
\leq  \, \, & n^r  \sum_{(Q_1,\ldots,Q_{k-2r}) \in S_n(\pi^{(r)})}  |\mathbb E\left[ a_n(Q_i) a_n(Q_{i+j})\right]|.
\end{align*}
At this point, we would like to use
\begin{equation}\label{false_est}
 |\mathbb E\left[ a_n(Q_i) a_n(Q_{i+j})\right]| = c_n(|q_i-q_{i+j}|),
\end{equation}
then calculate the number of points $(Q_1,\ldots,Q_{k-2r}) \in S_n(\pi^{(r)})$
for given $q_i,q_{i+j}$ and finally use (C1) to obtain 
\begin{equation*}
\sum_{q_i,q_{i+j}=1}^n  c_n(|q_i-q_{i+j}|)=o(n^2).
\end{equation*}
Unfortunately, (\ref{false_est}) is only valid for $q_i\leq q_{i+1}$ and $q_{i+j}\leq q_{i+j+1}$ or for  $q_i\geq q_{i+1}$ and  $q_{i+j}\geq q_{i+j+1}$ where $Q_i=(q_i,q_{i+1})$ and $Q_{i+j}=(q_{i+j},q_{i+j+1})$.
Hence, we have to take the ordering of the $q_i,q_{i+1}$ and $q_{i+j},q_{i+j+1}$ into account.
As before, we denote $Q_l=(q_l,q_{l+1})$ for $l=1,\ldots,k-2r$, where  $k-2r$ is identified with 1.
We distinguish two types of pairs $(Q_i,Q_{i+j})$: We call 
$$(Q_i,Q_{i+j})
\begin{cases}
\text{ positive,} & \text{if } \text{sgn}(q_i-q_{i+1}) = \text{sgn} (q_{i+j}-q_{i+j+1}) \\ \text{ negative,} & \text{otherwise}
\end{cases}.$$
Then we have 
$$|\text{Cov}(Q_i,Q_{i+j})| =
\begin{cases}
 c_n(|q_i-q_{i+j}|),& \text{if }  (Q_i,Q_{i+j}) \text{ positive}\  \\c_n(|q_i-q_{i+j+1}|),&  \text{if }  (Q_i,Q_{i+j}) \text{ negative} 
\end{cases}.$$
We claim the following: 
For given $q_i$ and $q_{i+j}$ there are less than $n^{\frac{k}{2} -r-1}$ tuples $(Q_1,\ldots,Q_{k-2r}) \in S_n(\pi^{(r)})$ with $ (Q_i,Q_{i+j})$ positive. Similarly, for given $q_i,q_{i+j+1}$ there are less than $n^{\frac{k}{2} -r-1}$ tuples $(Q_1,\ldots,Q_{k-2r}) \in S_n(\pi^{(r)})$ with $ (Q_i,Q_{i+j})$ negative.
We start with the case  $ (Q_i,Q_{i+j})$ positive and $q_i,q_{i+j}$ fixed. We have $n$ possible choices for $q_{i+1}$. By $i\sim i+j$, this determines the value of $q_{i+j+1}$ (recall $q_{i+j}$ is fixed) and hence $Q_i,Q_{i+j}$ are fixed. Since  $j$ is chosen to be minimal, the $j-1$ elements in $\{i+1,\ldots, i+j-1\}$ lie in $j-1$ different partition blocks. Hence, we have $n$ possibilities for each of the $j-2$ points $q_{i+2}, \ldots, q_{i+j-1}$. So far, there were $n \cdot n^{j-2}$ possibilities and we fixed $Q_i,\ldots, Q_{j+i}$.
We want to apply counting principle (CP2) to determine the number of possible choices for the remaining pairs $Q_{i+j+1}, \ldots, Q_{k-2r}, Q_1,\ldots, Q_{i-1}$. Hence, we have to determine the number of partition blocks occupied by $ {i+j+1}, \ldots,  {k-2r},  1,\ldots,  {i-1}$, that have an empty intersection with the set $\{i,\ldots,i+j\}$. From the total of $\frac{k}{2}-r$ partition blocks of $\pi^{(r)}$ one block is occupied by $i$ and $i+j$ and the $j-1$  blocks occupied by  $\{i+1,\ldots, i+j-1\}$ each contain one element in $\{{i+j+1}, \ldots,  {k-2r},  1,\ldots,  {i-1}\}$ as well. Thus, (CP2) gives $n^{\frac{k}{2}-r-1-(j-1)}=n^{\frac{k}{2}-r-j}$ possibilities to fix  $Q_{i+j+1}, \ldots, Q_{k-2r}, Q_1,\ldots, Q_{i-1}$.
Hence, for fixed $q_i,q_{i+j}$ we have a total of 
 $$n\cdot n^{j-2} n^{(\frac{k}{2}-r-j)}=n^{\frac{k}{2}-r-1} $$
 possibilities to choose  $Q_1,\ldots,Q_{k-2r}$.
 We obtain 
\begin{align}
&n^r  \sum_{(Q_1,\ldots,Q_{k-2r}) \in S_n(\pi^{(r)}) \atop Q_i,Q_{i+j} \text{ positive}}  |\mathbb E\left[ a_n(Q_i) a_n(Q_{i+j})\right]|
\nonumber \\
 \leq  \, \,& n^{\frac{k}{2}-1} \sum_{q_i,q_{i+j}=1}^n c_n(|q_i-q_{i+j}|) \leq n^{\frac{k}{2}} \sum_{\tau=0}^{n-1} c_n(\tau)=o(n^{\frac{k}{2}+1})
\label{est_2}
\end{align} 
by $\sum_{\tau=0}^{n-1} c_n(\tau) =o(n)$ (see (ii) in (C1)).

In the case $ (Q_i,Q_{i+j})$ negative and $q_i,q_{i+j+1}$ given, we have $n$ possibilities to fix $q_{i+1}$, determining $q_{i+j}$ by $i \sim j$. Hence, $Q_i,Q_{i+j}$ are fixed and we can proceed just as in the case $ (Q_i,Q_{i+j})$ positive and we obtain 
\begin{align}
&n^r  \sum_{(Q_1,\ldots,Q_{k-2r}) \in S_n(\pi^{(r)}) \atop Q_i,Q_{i+j} \text{ negative}}  |\mathbb E\left[ a_n(Q_i) a_n(Q_{i+j})\right]|
 =o(n^{\frac{k}{2}+1}) \label{est_1}
\end{align} 
Combining  (\ref{est_2}) and (\ref{est_1}) leads to 
\begin{equation*}
 n^r  \sum_{(Q_1,\ldots,Q_{k-2r}) \in S_n(\pi^{(r)})}\left| \mathbb E\left[ a_n(Q_1) a_n(Q_2) \ldots a_n(Q_{k-2r})\right]\right|=o(n^{\frac{k}{2}+1}),
\end{equation*}
completing the proof.
  \end{proof}
%
\subsection{The expected $k$-th moment of $\mu_n$ under (C2)}  \label{subsec:C1}
Before we can proceed to study the large $n$-limit of 
\begin{equation}
\frac{1}{n^{\frac{k}{2}  + 1}} \sum_{\pi \in \mathcal{CPP}(k) } \sum_{(P_1,\ldots,P_k) \in S_n(\pi)} \mathbb E\left[ a_n(P_1) a_n(P_2) \ldots a_n(P_k)\right]
\label{NC1}
\end{equation}
under condition (C2),
we state a combinatorial lemma needed for the proof. Throughout this section we write $P_i=(p_i,p_{i+1}), i=1,\ldots,k$. 
We want to pay special attention to pairs $P_i, P_j$ with $i\sim j$  and 
\begin{equation}
\label{comb_P_quer}
P_i=P_j \quad \text{or} \quad P_i=\overline{P}_j\coloneqq(p_{j+1},p_j).
\end{equation}
The lemma states that  if a block $\{i,j\}$ of a  partition $\pi$ with (\ref{comb_P_quer})
 is crossed by some other block, the number of points $(P_1,\ldots,P_k) \in S_n(\pi)$ is of order $o(n^{\frac{k}{2}+1})$.
\begin{lemma}[Crossing Property, cf.~Lemma 5.4 in \cite{FL2}]\label{crossing_prop}
Let $k\in \mathbb N$, $\pi \in \mathcal{PP}(k)$ and $i<j$ with $i\sim j$. 
Set 
\begin{equation*}
S_n(\pi, i,j)\coloneqq \{(P_1,\ldots,P_k) \in S_n(\pi) : P_i=P_j \text{ or } P_i=\overline{P}_j \}.  
\end{equation*}
If there exist $i',j'$ with $i' \sim j', i<i'<j$ and either $j'<i$ or $j<j'$ (i.e.~the block $\{i,j\}$ is crossed by the block $\{i',j'\}$), we have 
\begin{equation*}
\# S_n(\pi,i,j)=o(n^{\frac{k}{2}+1}).
\end{equation*}
\begin{proof}
To construct $(P_1,\ldots,P_k) \in S_n(\pi, i,j)$, first choose $p_i$ and $p_{i+1}$, each allowing for $n$ possibilities. Then $P_i$ is fixed and we choose one of the two possibilities $P_i=P_j$ or $P_i=\overline{P}_j$, fixing $P_j$. Let $r$ denote the number of partition blocks occupied by $\{i+1, \ldots, i'-1\} \cup \{j,\ldots , i'+1 \}$. By similar arguments as in (CP1)  we have less than $n^r$ choices to fix $P_{i+1},\ldots, P_{i'-1}, P_{j},\ldots, P_{i'+1}$. Hence, $P_{i'}$ is determined by consistency without any further choice. 
So far, we fixed $P_l$ for $l$ in $r+2$ different partition blocks. By (CP2) there are at most  $n^{\frac{k}{2}-r-2}$ choices to fix all remaining points $P_l$. In total, there are $n^{\frac{k}{2} }$ possibilities to construct $(P_1,\ldots,P_k)  \in S_n(\pi, i,j).$
\end{proof}
\end{lemma}
We continue by considering the term in (\ref{NC1}) and observe that for $\pi \in \mathcal{CPP}(k) $ and $(P_1,\ldots,P_k) \in S_n(\pi)$ the term  $\mathbb E\left[ a_n(P_1) a_n(P_2) \ldots a_n(P_k)\right]$ is a product of factors
\begin{equation} \label{Darst_E}
\mathbb E\left[ a_n(P_i) a_n(P_j)\right]= \begin{cases} 1, & \text{ if } P_i=P_j \text{ or }P_i= \overline{P}_j 
\\
c_n, & \text{ else}
\end{cases}, \quad  i \sim j.
\end{equation}
To account for this fact, we introduce the following notation for a given partition $\pi \in \mathcal{CPP}(k)$ and $(P_1,\ldots, P_k) \in S_n(\pi)$:
\begin{equation*}
m(P_1,\ldots,P_k)\coloneqq \# \{ 1 \leq i<j\leq k : P_i=P_j \text{ or }P_i=\overline{P}_j \} \leq \frac{k}{2}.
\end{equation*}
For $l\in \{1,\ldots, \frac{k}{2}\}$ we set 
\begin{equation*}
A_n^{(l)}(\pi)\coloneqq \{(P_1,\ldots,P_k) \in S_n(\pi) : m(P_1,\ldots,P_k)=l\}.
\end{equation*}
Hence, we can write for $\pi \in \mathcal{CPP}(k)$
\begin{align*}
\frac{1}{n^{\frac{k}{2}  + 1}} \sum_{(P_1,\ldots,P_k) \atop \in S_n(\pi)} \mathbb E\left[ a_n(P_1) a_n(P_2) \ldots a_n(P_k)\right]=\frac{1}{n^{\frac{k}{2}  + 1}} \sum_{l=0}^{\frac{k}{2}}  c_n^{\frac{k}{2}-l} \# A_n^{(l)}(\pi).
\end{align*}
Moreover, we set 
\begin{align*}
B_n^{(l)}(\pi)& \coloneqq  \{(P_1,\ldots,P_k) \in S_n(\pi) : m(P_1,\ldots,P_k)=l; \\
&P_i=P_j \text{ or } P_i=\overline{P}_j, i<j\,\, \Rightarrow \, \, j=i+1 \text{ or } \pi|_{\{i+1,\ldots, j-1\}} \text{ is a pair partition}\}.
\end{align*}
By the crossing property of Lemma \ref{crossing_prop} we have
\begin{equation} \label{B_n_l:1}
\frac{1}{n^{\frac{k}{2}  + 1}}  \#\left(A_n^{(l)}(\pi) \setminus B_n^{(l)}(\pi)\right) \to 0 , \quad  n \to \infty.
\end{equation}
Indeed, any point $(P_1,\ldots,P_k) \in \left(  A_n^{(l)}(\pi) \setminus B_n^{(l)}(\pi)\right)  $ belongs to some $S_n(\pi, i,j)$ and 
\begin{equation*}
\#\bigcup_{i,j \in \{1,\ldots,k\}} S_n(\pi, i,j) = o(n^{\frac{k}{2}+1}).
\end{equation*}
In order to show  that $n^{-(\frac{k}{2}  + 1)} \#B_n^{(l)}(\pi)$ vanishes for almost all values of $l$, we introduce the notion of \textit{height} of a pair partition $\pi \in \mathcal{PP}(k)$:
\begin{equation*}
h(\pi)\coloneqq \# \{ 1\leq i<j\leq k, i\sim j: j=i+1 \text{ or } \pi|_{\{i+1,\ldots, j-1\}} \text{ is a pair partition}\}.
\end{equation*}
Since 
\begin{equation*}
\left( (P_i=\overline{P}_j) \text{ or }  (P_i=P_j)  \right) \quad  \Rightarrow  \quad i \sim j,
 \end{equation*}
we have 
\begin{equation}\label{B_n_l:2}
 B_n^{(l)}(\pi) =\emptyset \quad \text{for } l>h(\pi).
\end{equation}
Combining (\ref{B_n_l:1}) and (\ref{B_n_l:2})
leads to 
 \begin{align}\label{B_n_l:3}
\frac{1}{n^{\frac{k}{2}  + 1}} \sum_{l=0}^{\frac{k}{2}}  c_n^{\frac{k}{2}-l} \# A_n^{(l)}(\pi)=\frac{1}{n^{\frac{k}{2}  + 1}} \sum_{l=0}^{h(\pi)}  c_n^{\frac{k}{2}-l} \# B_n^{(l)}(\pi) + o(1).
\end{align}
It is a consequence of the following lemma that only $B_n^{(h(\pi))}(\pi)$ gives a non-vanishing contribution in (\ref{B_n_l:3}).
\begin{lemma}\label{lemma_B_n}
For $k \in \mathbb N$, $\pi \in \mathcal{PP}(k)$ and $l <h(\pi)$ we have
\begin{equation*}
\lim_{n \to \infty} \frac{1}{n^{\frac{k}{2}  + 1}}  \# B_n^{(l)}(\pi) =0.
\end{equation*}
\end{lemma}
\begin{proof}
Let  $\pi \in \mathcal{PP}(k)$ and $l <h(\pi)$. We want to construct $(P_1,\ldots, P_k) \in  B_n^{(l)}(\pi) $. As we assume   $l <h(\pi)$, there are indices $i,j$ that give a contribution to   $h(\pi)$ but the corresponding pairs $P_i,P_j$ do not contribute  to $m(P_1,\ldots, P_k) $, i.e.~there exist $1\leq i<j \leq k$ such that 
\begin{enumerate}
\item[(i)] $i \sim j$,
\item[(ii)] $j=i+1$ or $\pi|_{\{i+1,\ldots,j-1\}}$ is a pair partition,
\item[(iii)] $P_i \neq P_j$ and $P_i \neq \overline{P}_j.$
\end{enumerate}
In particular, $j=i+1$ cannot be satisfied, as this implies $P_i=\overline{P}_j$. Hence, we can assume $j>i+1$, $\pi|_{\{i+1,\ldots,j-1\}}$ is a pair partition (with $(j-i-1)/2$ blocks)
 and, as a consequence of (iii), $p_{i+1} \neq p_j$.
 We observe that  
 $$\# (\pi|_{\{1,\ldots,k\}\setminus\{i+1,\ldots,j-1\}})=\frac{k}{2}-\frac{j-i-1}{2}.$$
Then there are  $n$ possibilities to choose $p_{i+1}$ and by (CP1) we have $n^{\frac{k}{2}-\frac{j-i-1}{2}}$ possibilities to successively choose $p_i,p_{i-1},\ldots, p_1,p_k, \ldots, p_j$.
  Applying (CP2) to choose $p_{i+2}, \ldots, p_{j-1}$ would amount to $n^\frac{j-i-1}{2}$ possibilities, but we claim that there are actually only $Cn^{\frac{j-i-1}{2}-1}$ possibilities.
Recalling that  $ p_{i+1},p_j$ are already known and distinct, we observe that we have  
 \begin{equation}\label{alt_summe}
 0 \neq p_{i+1} -p_j =\sum_{s=1}^{j-i-1}(-1)^{s}(p_{i+s}+p_{i+s+1}).
 \end{equation}
As $\pi|_{\{i+1,\ldots,j-1\}}$ is a pair partition, neglecting their sign, each term $p_{i+s}+p_{i+s+1}$ appears exactly twice in  the alternating sum in (\ref{alt_summe})  and as the sum does not vanish, there are $1\leq \alpha,\beta \leq j-i-1$ with $i+\alpha \sim i+\beta$ and $(-1)^{\alpha}=(-1)^{\beta}$.
Then we have
 \begin{equation}\label{alt_summe:2}
 p_{i+1} -p_j =2(-1)^{\alpha} (p_{i+\alpha} + p_{i+\alpha+1})+ \sum_{s=1,\ldots,j-i-1 \atop s\neq \alpha,\beta} (-1)^{s}(p_{i+s}+p_{i+s+1}).
 \end{equation}
 For each of the $\frac{j-i-1}{2}-1$ blocks $\{r,s\} \subset \{i+1,\ldots,j-1\}\setminus \{ i+\alpha,i+\beta\}$ we assign one of $2n$ possible values to $p_{r}+p_{r+1}$ (and hence to $p_s+p_{s+1}$), amounting to $(2n)^{\frac{j-i-1}{2}-1}$ possibilities. Then the alternating sum in (\ref{alt_summe:2}) is fixed and  as we already know $ p_{i+1},p_j$, we can calculate $(p_{i+\alpha} + p_{i+\alpha+1})$ and hence $p_{i+\beta} + p_{i+\beta+1}$. Knowing $p_{i+1}, p_j$ and all the terms $p_{l}+p_{l+1}$, $l=i+1,\ldots, j-1$, the values of $p_{i+2},\ldots,p_{j-1}$ are uniquely determined. 
 Hence, there was a total of 
 $$n^{\frac{k}{2}-\frac{j-i-1}{2}+1}(2n)^{\frac{j-i-1}{2}-1} =C_{i,j}n^{\frac{k}{2}} $$
possibilities to choose  $(P_1,\ldots, P_k) \in  B_n^{(l)}(\pi)$, where $C_{i,j}$ denotes some constant that may depend on $i$ and $j$ only.
Thus, we have
\begin{equation*}
0 \leq \lim_{n \to \infty} \frac{1}{n^{\frac{k}{2}  + 1}}  \# B_n^{(l)}(\pi) \leq  \lim_{n \to \infty} C_{i,j}\frac{n^{\frac{k}{2}}}{n^{\frac{k}{2}  + 1}} =0,
\end{equation*}
completing the proof.
 \end{proof}
 So far, we showed that 
 \begin{align*}
&\frac{1}{n^{\frac{k}{2}  + 1}} \sum_{\pi \in \mathcal{CPP}(k) } \sum_{(P_1,\ldots,P_k) \in S_n(\pi)} \mathbb E\left[ a_n(P_1) a_n(P_2) \ldots a_n(P_k)\right]\\
=&\frac{1}{n^{\frac{k}{2}  + 1}} \sum_{\pi \in \mathcal{CPP}(k) }   c_n^{\frac{k}{2}-h(\pi)} \# B_n^{(h(\pi))}(\pi) + o(1).
\end{align*}
Observe that we have  by Lemma \ref{lemma_B_n} and Lemma \ref{crossing_prop}
\begin{align*}
&\lim_{n \to \infty}\frac{1}{n^{\frac{k}{2}  + 1}}\# B_n^{(h(\pi))}(\pi)=\lim_{n \to \infty}\frac{1}{n^{\frac{k}{2}  + 1}} \#\left(B_n^{(h(\pi))}(\pi)\cup \left( \bigcup_{l < h(\pi)} B_n^{(l)}(\pi)\right) \right)\\
=&\lim_{n \to \infty}\frac{1}{n^{\frac{k}{2}  + 1}} \# \{ \{(P_1,\ldots,P_k) \in S_n(\pi) : \\
& \quad \quad P_i=P_j \text{ or } P_i=\overline{P}_j, i<j \Rightarrow j=i+1 \text{ or } \pi|_{\{i+1,\ldots, j-1\}} \text{ is a pair partition}\}. \}
\\ = &\lim_{n \to \infty}\frac{1}{n^{\frac{k}{2}  + 1}} \# S_n(\pi)
\end{align*}
In order to state a result about the limit of $\frac{1}{n^{\frac{k}{2}  + 1}}\#S_n(\pi)$, we introduce the notion of Hankel volumes.
\begin{definition}
For a pair partition  $\pi \in \mathcal{PP}(k)$ we consider the set of equations in the $k+1$ variables $x_0,\ldots,x_k \in [0,1]$:
\begin{align*}
x_1 +x_0  &= x_{l_1}+ x_{l_1-1}, \quad \text{ if } 1 \sim l_1 \\
x_2 + x_1 &=  x_{l_2} + x_{l_2-1}, \quad \text{ if } 2 \sim l_2  \\
\ldots \\
x_i + x_{i-1} &=  x_{l_i} + x_{l_i-1}, \quad\, \, \text{ if } i \sim l_i \\
\ldots 
\\
x_k + x_{k-1} &=  x_{l_k} + x_{l_k-1}, \quad  \text{ if } k \sim l_k.
\end{align*}
Note that, as $\pi$ is a pair partition, these are actually only $\frac{k}{2} $ equations. 
If $\pi = \{ \{i_1,j_1\}, \{i_2,j_2\}, \ldots \{ i_{\frac{k}{2} },j_{\frac{k}{2} }\}\}$  with $i_l<j_l,l=1,\ldots, \frac{k}{2} $, we solve the equations for $x_{j_1},\ldots, x_{j_{k/2}}$, determining a cross section of the cube $[0,1]^{\frac{k}{2}+1}$. The volume of this cube is denoted by $p_H(\pi)$.
It is a result of \cite{BDJ} that $p_H(\pi)$ is exactly the limit of $\frac{1}{n^{\frac{k}{2}  + 1}} \# S_n(\pi)$.
\end{definition}
\begin{proposition}[Lemma 4.8 in \cite{BDJ}]
For $k \in \mathbb N$, $\pi \in \mathcal{PP}(k)$ we have
$$\lim_{n \to \infty}\frac{1}{n^{\frac{k}{2}  + 1}} \# S_n(\pi)=p_H(\pi).$$
\end{proposition}
Finally, using that under (C2) we have $c_n \to c$ as $n$ tends to infinity, we have shown that 
\begin{align} 
\lim_{n \to \infty}\frac{1}{n} \mathbb E \left[\tr X_n^k \right]
&=C_{\frac{k}{2}}+   \sum_{\pi \in \mathcal{CPP}(k) }     c^{\frac{k}{2}-h(\pi)} p_H(\pi).
\nonumber \\
&=  \sum_{\pi \in \mathcal{PP}(k) }     c^{\frac{k}{2}-h(\pi)} p_H(\pi) =: M_k.
\label{moments}
\end{align}
The second equality is due to the fact that all statements  in subsection \ref{subsec:C1} remain valid for all pair partitions that are not necessarily crossing. 
%
%
\section{The Proofs of Theorem \ref{MainTheorem_2}, Theorem \ref{MainTheorem_1} and  Theorem \ref{theo_limiting_distr}} \label{sec_completing_proofs}
To complete the proof of Theorems \ref{MainTheorem_2} and   \ref{MainTheorem_1}  we need to show that, with probability one, the empirical spectral distribution converges weakly to a non-random limit under (C1) resp.~under (C2). 
Under condition (C1) the limiting measure is Wigner's semi-circle $\mu$ (see (\ref{semi_circ})), which is uniquely determined by its moments: the odd moments vanish and the $2k$-th moment is given by the $k$-th Catalan number $C_k$ (see e.g.~Section 2.1.1 in \cite{anderson}).
If (C2) is valid, we will see that the limiting measure  $v_c$ is determined by its moments
\begin{equation}
\label{moments:vc}
\int x^k dv_c = \begin{cases}  0 & k \text{ odd} \\ M_k &k \text{ even} \end{cases}.
\end{equation}
The fact that these moments uniquely determine $v_c$ can be verified by checking the Carleman condition. 
\begin{remark}\label{Remark_lim_distr}
From (\ref{moments}) and (\ref{moments:vc}) we can already deduce the following facts about the measure $v_c$:
\begin{enumerate}
\item[(i)]
The measure $v_0$ is the semi-circle distribution. Indeed, for $c=0$ we have $M_k=C_{k/2}$ for $k$ even, which are exactly the moments of the semi-circle distribution. 
\item[(ii)]
The measure $v_1$ equals $\gamma_H$, as $M_k=  \sum_{\pi \in \mathcal{PP}(k) }     p_H(\pi)$ for $c=1$, which are exactly the moments of $\gamma_H$ (cf.~\cite{BDJ}).
\item[(iii)]
The measure $v_c$ is symmetric for all $c \in [0,1]$ because all odd moments vanish.
\item[(iv)] For $0<c \leq 1$ the measure $v_c$ has unbounded support. 
Indeed, as a bounded support of $v_c$ would lead to  $M_{2k}\leq C^{2k}$, it suffices to verify 
\begin{equation*}
\limsup_{k \to \infty} (M_{2k})^{\frac{1}{k}} =\infty.
\end{equation*}
The above relation  is a consequence of 
$c^{\frac{k}{2}} \int x^k d\gamma_H(d) \leq M_k$
and (see Proposition A.2 in \cite{BDJ})
$\limsup_{k \to \infty}(\int x^{2k} d\gamma_H(d))^{1/k} =\infty.$
 \end{enumerate}
\end{remark}
We can now prove Theorems \ref{MainTheorem_2} and  \ref{MainTheorem_1}, which read:
\begin{theorem}\text{} Let $\mu$ be Wigner semi-circle given in (\ref{semi_circ}) and let $v_c$ be the measure that is uniquely determined by its moments according to (\ref{moments:vc}).  
\begin{enumerate}
\item[(i)]Under condition (C1), $\mu_n$ converges for $n \to~\infty$ weakly, with probability one, to $\mu$ . 
\item[(ii)]
Under condition (C2), $\mu_n$  converges for $n \to~\infty$ weakly, with probability one, to $v_c$.
\end{enumerate}
\end{theorem}
\begin{proof}
We will use the concentration inequality obtained in Proposition 4.9 in \cite{BDJ}, which can easily be extended to the case of matrices with independent skew-diagonals analogue to Lemma 3.5 in \cite{FL1}. Hence, we have under both conditions (C1) and (C2)
\begin{equation}
\label{conc_ineq}
\mathbb E \left[ \left( \tr(X_n^k)-\mathbb E(\tr(X_n^k))^4 \right) \right] \leq C n^2, \quad \forall k\in\mathbb N.
\end{equation}
As the limiting distributions $\mu$ and $v_c$ are uniquely determined by their moments, it suffices to show 
\begin{equation*}
 \int x^k d\mu_n(X_n)  = \frac{1}{n}  \tr (X_n^k) \to \mathbb E[Y^k], \quad n\to \infty\quad \text{ almost surely},
\end{equation*}
where $Y$ denotes a random variable distributed according to $v_c$ if we consider (C2) and according to $\mu$ if we consider (C1).
By Chebyshev's inequality and (\ref{conc_ineq}) we have  for $\varepsilon >0, k,n \in \mathbb N$
\begin{equation*}
\mathbb P \left( \left|  \frac{1}{n}  \tr (X_n^k) - \mathbb E \left(  \frac{1}{n}  \tr (X_n^k)\right) \right|> \varepsilon \right) \leq \frac{C}{\varepsilon^4 n^2}.
\end{equation*}
By the Borel-Cantelli Lemma we obtain 
\begin{equation*}
\frac{1}{n}  \tr (X_n^k)-\mathbb E \left(  \frac{1}{n}  \tr (X_n^k)\right) \to 0 \quad n\to\infty  \quad \text{ almost surely}.
\end{equation*}
Since it is the result of the previous sections that 
\begin{equation*}
\mathbb E \left(  \frac{1}{n}  \tr (X_n^k)\right) \to \mathbb E[Y^k], \quad n\to\infty,
\end{equation*}
this completes the proof.
\end{proof}
It remains to prove the claims about the limiting measure $v_c$, which are listed in Theorem \ref{theo_limiting_distr}. Some of those are already stated in Remark \ref{Remark_lim_distr} and it suffices to show the following Lemma.
\begin{lemma}[cf.~Lemma 6.2 in \cite{FL2}]
With the notation of Theorem \ref{theo_limiting_distr} we have
$$v_c=v_{0,1-c} \boxplus v_{1,c}, \quad 0 \leq c \leq 1.$$
Moreover, $v_c$ has a smooth density if $0 \leq  c < 1$.
\end{lemma}
\begin{proof}
 Recall that  $v_{0,1-c} $ denotes the rescaled semi-circle  with variance $1-c$ and  $v_{1,c}$ the rescaled Hankel distribution $\gamma_H$ with variance c as derived in \cite{BDJ}. It suffices to show that the
free cumulants of the free convolution of $v_{0,1-c} $ and  $v_{1,c}$  coincide with the
free cumulants of $v_c$.
We apply the same arguments as in  Lemma 6.2 in \cite{FL2} (only replacing $p_T$ by $p_H$), that rely on the results of  Lemma A.4 in \cite{BDJ} (see also p. 152 in \cite{Speicher}). 
Similarly, we conclude
$$(1-c)^k \kappa_{2k}(\mu) + c^k \kappa_{2k}(\gamma_H)=\kappa_{2k}(v_c),$$
proving 
$$v_c=v_{0,1-c} \boxplus v_{1,c}, \quad 0 \leq c \leq 1.$$ 
Using general results about the free convolution with the semi-circle distribution provided in \cite{biane}, we obtain that $v_c$ has a smooth density for $0\leq c<1$. 
\end{proof}
\begin{remark}
The  boundedness of the density of $v_c$ is not derived here as   the boundedness of $\gamma_H$ is not yet 
 available in the literature. Note that 
for matrices with independent diagonals  the boundedness of the corresponding density  could be derived in \cite{FL2}, using \cite{biane} and the boundedness of $\gamma_T$ \cite{virag}.
\end{remark}
As already noted, our line of arguments follows  \cite{FL2,FL1} and in the following concluding remark we comment on the differences between  the proofs.
\begin{remark}\label{remark_diff_toeplitzcase}
The main difference between the ensembles considered here and those considered in \cite{FL2,FL1}  is that instead of (\ref{equivalence_relation}) the dependence structure of the matrix entries in \cite{FL2,FL1} is given by 
\begin{equation*}  
|p_i-q_i| = |p_j-q_j|\quad \Leftrightarrow \quad  i\sim j.
\end{equation*}
Hence, the validity of all arguments from \cite{FL2,FL1}  has to be verified for  (\ref{equivalence_relation}).

The proof of Theorem \ref{MainTheorem_2} follows the corresponding proofs in \cite{FL1}.  Here, we do not introduce the sets $S_n^*(\pi) \subset S_n(\pi)$ (there is actually no natural way to do this in our setting),  and all needed relations have to be derived from (\ref{equivalence_relation}) directly. The lack of $S_n^*$ requires the distinction of positive and negative pairs in the proof of Lemma \ref{Lemma_4}.

The proof of Theorem \ref{MainTheorem_1} (corresponding to \cite{FL2}) requires more modifications. Again, we do not introduce the sets $S_n^*(\pi)$, but in this case the implications are more severe. The most prominent one is that the analogue of  \cite[Lemma 5.5]{FL1} is not valid  and it has to be replaced by Lemma  \ref{lemma_B_n} of this paper (i.e.~$\# B_n^{(l)}(\pi)$ is not necessarily zero for all $n$ but vanishes in the limit  $n \to \infty$), which required new ideas. Moreover, in Section \ref{subsec:C1} we have to additionally consider the pairs $\overline{P}_j$ in the definition of $S_n(\pi, i,j)$ in Lemma \ref{crossing_prop}, in (\ref{Darst_E}) and in all derived terms such as $m(P_1,\ldots,P_k)$ and $B_n^{(l)}(\pi)$. Hence, we have to  verify that the required estimates remain valid. 

In both cases, the extension from the  convergence of the expected $k$-th moment to the almost sure convergence and the proof of Theorem \ref{theo_limiting_distr} can  be carried out analogously to \cite{FL2,FL1}. 
\end{remark}

\section*{Acknowledgments}

The author would like to thank Matthias L\"owe for drawing her attention to
the subject of this paper and for many fruitful discussions.

\bibliographystyle{plain}
\bibliography{bibliography}

\end{document}